\documentclass{amsart}

\usepackage{amssymb,amsfonts}
\usepackage{amsmath,amscd}
\usepackage[all,arc]{xy}
\usepackage{enumerate}
\usepackage{mathrsfs}
\usepackage[pdftex]{graphicx}
\usepackage[utf8]{inputenc}
\usepackage[T1]{fontenc}
\usepackage[usenames,dvipsnames]{color}
\usepackage[bookmarks=false]{hyperref}
\usepackage{dsfont}
\usepackage{tikz}
\usetikzlibrary{automata,positioning}
\usepackage{multirow}
\usepackage{scalerel}
\theoremstyle{plain}
\newtheorem{thm}{Theorem}[section]
\newtheorem{cor}[thm]{Corollary}
\newtheorem{prop}[thm]{Proposition}
\newtheorem{lem}[thm]{Lemma}
\newtheorem{conj}[thm]{Conjecture}

\theoremstyle{definition}
\newtheorem{defn}[thm]{Definition}

\newtheorem{question}[thm]{Question}

\theoremstyle{remark}
\newtheorem{rmk}[thm]{Remark}

\newcommand{\bbC}{\mathbb{C}} 
\newcommand{\bbD}{\mathbb{D}}

\newcommand{\bbN}{\mathbb{N}} 

\newcommand{\bbR}{\mathbb{R}} 

\newcommand*{\defeq}{\mathrel{\vcenter{\baselineskip0.5ex \lineskiplimit0pt
			\hbox{\scriptsize.}\hbox{\scriptsize.}}}%
	=}

\numberwithin{equation}{section}

\title{A Riemannian metric on polynomial hyperbolic components}

\author{Yan Mary He and Hongming Nie}
\address{Department of Mathematics, University of Toronto, M5S 2E4, Canada \&
Math Research Unit, University of Luxembourg, L-4364, Luxembourg}
\email{yanmary.he@mail.utoronto.ca}

\address{Einstein Institute of Mathematics, The Hebrew University of Jerusalem, 9190401, Israel}
\email{hongming.nie@mail.huji.ac.il}
\date{\today}

\begin{document}

\maketitle
\begin{abstract}
 We introduce a Riemannian metric on certain hyperbolic components in the moduli space of degree $d \ge 2$ polynomials. Our metric is constructed by considering the measure-theoretic entropy of a polynomial with respect to some equilibrium state. As applications, we show that the Hausdorff dimension function has no local maximum on such hyperbolic components. We also give a sufficient condition for a point not being a critical point of the Hausdorff dimension function.
\end{abstract}

\section{Introduction}\label{intro}
For $d\ge 2$, the parameter space $\mathrm{Poly}_d$ is the space of degree $d$ polynomials, and the moduli space $\mathrm{poly}_d$ is the space of affine conjugacy classes of degree $d$ polynomials.
A polynomial $P\in\mathrm{Poly}_d$ is \emph{hyperbolic} if all the critical points under iterations converge to attracting cycles. The space of degree $d$ hyperbolic polynomials is open in $\mathrm{Poly}_d$. Each of its components is a \emph{hyperbolic component in $\mathrm{Poly}_d$}.  Moreover, it descends to an open subspace in $\mathrm{poly}_d$ and the corresponding component is a \emph{hyperbolic component in $\mathrm{poly}_d$}. 

The goal of this paper is to introduce a natural metric on certain hyperbolic components in $\mathrm{poly}_d$. Using the measure-theoretic entropy of $P$ with respect to some equilibrium state, we construct a non-negative two-form on any hyperbolic component in $\mathrm{poly}_d$. We prove that this two-form is positive-definite on certain hyperbolic components. Most of our argument works also for hyperbolic components in the moduli space $\mathrm{rat}_d$ of degree $d$ rational maps with additional assumptions. As applications, we study critical points of the Hausdorff dimension function $\delta : \mathcal{H} \to (0,2)$, sending $[P]$ to the Hausdorff dimension of the Julia set of $P$, where $\mathcal{H}$ is a hyperbolic component in $\mathrm{poly}_d$.

Now we turn to the details of the above ingredients and the statement of results of this paper.

\subsection{Statement of results}
Denote by $\mathrm{Rat}_d$ the space of degree $d\ge 2$ rational maps and let $\widetilde{\mathcal{H}}$ be a hyperbolic component in $\mathrm{Rat}_d$. For $f\in \widetilde{\mathcal{H}}$, let $\Omega(J(f))$ be the space of positive $f$-invariant probability measure on the Julia set $J(f)$. 
A probability measure $\mu \in \Omega(J(f))$ is called a {\it primitive orbit measure} if it is uniformly supported on a periodic cycle in $J(f)$. If we equip $\Omega(J(f))$ with the weak-$*$ topology, the set of primitive orbit measures is dense in $\Omega(J(f))$ (see Proposition \ref{density}). Then there exists a unique continuous function $M_{f} : \Omega(J(f)) \to \bbR$ such that $$M_f(\mu) = \frac{1}{k}\log \big| (f^k)'(x) \big|$$ if $\mu\in \Omega(J(f))$ is a primitive orbit measure supported on a period $k$ periodic orbit $\hat x= \{x, \cdots, f^{k-1}(x) \}$. In fact, $M_f(\mu)$ is the Lyapunov exponent of $f$ with respect to $\mu$.

Since $\widetilde{\mathcal{H}}$ is a hyperbolic component, consider the natural holomorphic motion of the Julia sets. There exists a neighborhood $U(f)$ of the map $f$ in $\widetilde{\mathcal{H}}$ such that for all $g\in U(f)$, the motion induces a unique homeomorphism $\phi_g : J(f) \to J(g)$ conjugating the dynamics $f : J(f) \to J(f)$ to $g : J(g) \to J(g)$. We define a map $M : \Omega(J(f)) \times U(f) \to \bbR$ by
\begin{equation*} 
M(\mu, g) = M_{g} \left((\phi_g)_* \mu \right).
\end{equation*}
It turns out that the map $M$ is real analytic in the second coordinate. Moreover, regarded as a map from $\Omega(J(f))$ to $C^{\infty}(U(f), \bbR)$, the map $M$ is continuous, see Proposition \ref{thm_Manalytic}. Given any $\mu \in \Omega(J(f))$, we define the {\it multiplier function} associated to $\mu$
$$M_{\mu} : U(f) \to \bbR$$ by $M_{\mu}(g): = M(\mu, g)$. This map $M_{\mu}$ is harmonic and in particular it is real-analytic, also see Proposition \ref{thm_Manalytic}.

Moreover, we analogously define the map $\widetilde{M} :  \Omega(J(f))\times U(f) \to \bbC$. We show that the map $\widetilde{M}_{\mu} :U(f)\to \bbC$ is holomorphic and the map $\mu \mapsto \widetilde{M}_{\mu}$ is continuous with respect to the topology of uniform convergence on compact subsets. We call $\widetilde{M}_{\mu}$ the \emph{complex multiplier function} (see Section \ref{sec_cxmul}).

Consider the pressure zero and H\"older continuous function $-\delta(f)\log|f'| : J(f) \to \bbR$. Let $\nu$ be its {\it equilibrium state}. Equivalently, $\nu$ is the unique $f$-invariant probability measure in the class of Hausdorff $\delta$ dimensional measures. Then $\nu \in \Omega(J(f))$. Hence the multiplier function $M_{\nu}: U(f) \to \bbR$ is real-analytic. On the other hand, the Hausdorff dimension function $\delta :  \widetilde{\mathcal{H}}\to \bbR$, sending $f$ to the Hausdorff dimension of $J(f)$ is also real-analytic \cite{Bowen79}.

To construct our metric, we consider the entropy function $G_{f}:U(f) \to \bbR$ given by $$G_{f}(g) = \delta(g) M_{\nu}(g).$$
It turns out that $G_f(f)\le G_f(g)$, see Proposition \ref{prop_metric}.


Hence the Hessian of $G_{f}$ at $g = f\in U(f)$ is well-defined. It defines a symmetric bilinear form $||\cdot||_G$ on the tangent space $T_{f}\widetilde{\mathcal{H}}$ as follows. For $t \in (-1, 1)$, let $\tilde{c}(t)$ be a path in $U(f)$ with $\tilde{c}(0) = f$ and $\tilde{c}'(0) = \tilde{v} \in T_{f}\widetilde{\mathcal{H}}$.
Then $$||\tilde{v}||_G^2 = \frac{\partial^2 G_{f}}{\partial{\tilde{v}} \partial\tilde{v} } = \frac{d^2}{dt^2} \bigg|_{t=0} G_{f}(\tilde{c}(t)).$$
It descents a $2$-form $||\cdot||_G$ on the corresponding hyperbolic component in $\mathrm{rat}_d$, see Section \ref{sec:2-form}. 

Restricting our attention to hyperbolic components in $\mathrm{poly}_d$, we show that $||\cdot||_G$ is positive-definite on certain hyperbolic components. Recall that the \emph{central component $\mathcal{H}_0$} is the hyperbolic component in $\mathrm{poly}_d$ containing the affine conjugacy class of $z^d$ and the \emph{shift locus $\mathcal{S}_d$} is the hyperbolic component in $\mathrm{poly}_d$ such that for each element $[P]$, all critical points of $P$ are in the basin of $\infty$.
Our main theorem is the following.
\begin{thm}\label{main}
Let $\mathcal{H}$ be a hyperbolic component in $\mathrm{poly}_d$ such that $\mathcal{H}$ is neither $\mathcal{H}_0$ nor $\mathcal{S}_d$. Then  on $\mathcal{H}$,
the metric $||\cdot||_G$ is a Riemannian metric and is conformal equivalent to the standard  pressure metric.
\end{thm}

For quadratic polynomials, we also obtain the above result for $\mathcal{H}_0$, see Theorem \ref{d=2}.
It would be interesting to study the properties of this metric further such as the K\"ahlerness, completeness, etc. On the other hand, in \cite{McMullen08}, McMullen introduced an analogue of the Weil-Petersson metric on the space of degree $d \ge 2$ Blaschke products. This space, via the Bers embedding, is isomorphic to $\mathcal{H}_0$ in $\mathrm{poly}_d$.  For $d=2$, it would also be interesting to investigate the relations between our metric and McMullen's metric.

\subsection{Applications to Hausdorff dimension}
This work originated from the works \cite{Bodart96} and \cite{He18}. In \cite{Bodart96}, Bodart and Zinsenister gave a numerical plot of the Hausdorff dimension function over the Mandelbrot set. In \cite{He18}, the first author obtained a numerical plot of the Hausdorff dimension one locus in the complement of the Mandelbrot set. This plots shows that the set of quadratic polynomials whose Julia sets have Hausdorff dimensions greater than one is star-like centered at $0$. These results might suggest some monotonicity properties of Hausdorff dimension. More results about Hausdorff dimensions for real quadratic polynomials are obtained in \cite{Douady97, Havard00, Jaksztas11, McMullen00, Ruelle82}.  
For algorithms which efficiently compute Hausdorff dimensions, we refer \cite{Oliver02, McMullen98}.

The above results motivated us to study the critical points of Hausdorff dimension function $\delta$. Our first result towards this direction concerns local maxima of $\delta$, which relates to Ransford's result in \cite{Ransford93}. Our method is more geometric compared to Ransford's analytic techniques.
\begin{thm} \label{thm_introapp1}
	Let $\mathcal{H}\subset\mathrm{poly}_d\setminus\mathcal{S}_d$ be a hyperbolic component. The Hausdorff dimension function $\delta: \mathcal{H} \to (0,2)$ has no local maximum on $\mathcal{H}$.
\end{thm}

Our next result relates critical points of $\delta$ to the multipliers of periodic cycles.
\begin{thm} \label{thm_introapp2}
	Let $\mathcal{H} \subset\mathrm{rat}_d$ be a hyperbolic component. Then $[f_0] \in\mathcal{H}$ is not a critical point of $\delta$ if 
	\begin{align*}
	\inf_{\{\hat{x}_n\}_{n \ge 1}} \displaystyle \liminf_{n \to \infty} \frac{1}{n} \left|\frac{D\lambda_{\hat{x}_n}([f_0])}{\lambda_{\hat{x}_n}([f_0])}\right| \neq 0
	\end{align*}
	where $\hat{x}_n$ is an $n$-cycle of $f_0$ and $\lambda_{\hat{x}_n}: \mathcal{H} \to \bbC$ is the function sending $[f]$ to the multiplier of $\phi_f(\hat{x}_n)$ where $f_0$ and $f$ are in the same hyperbolic component in $Rat_d$. The map $D\lambda_{\hat{x}_n}([f_0]): T_{[f_0]}\mathcal{H} \to \bbC$ is the differential of $\lambda_{\hat{x}_n}$. The infimum is taken over all the sequences of cycles $\hat{x}_n$.
\end{thm}

Our work is vastly inspired by the work of Bridgeman \cite{Bridgeman10} and Bridgeman-Taylor \cite{Bridgeman08} which established in a similar fashion an extension of the Weil-Peterson metric to the quasi-Fuchsian space $QF(S)$ of a closed surface $S$ of genus at least $2$. In particular, if $\Gamma \in QF(S)$ and $\delta(\Gamma)$ is the Hausdorff dimension of the limit set of $\Gamma$, Bridgeman proved in \cite{Bridgeman10} that the Hausdorff dimension function $\delta : QF(S) \to \bbR$ does not admit any local maximum in $QF(S)$. Our sufficient condition may shed light on proving the (non)existence of critical points of $\delta$ in $QF(S)$.

\subsection{Organization of the paper} The paper is organized as follows. In Section \ref{sec_mul}, we prove some preparatory results regarding the Hausdorff dimension of Julia sets (Proposition \ref{prop_HDJ}), the invariant measures (Proposition \ref{density}) and some properties of multipliers (Proposition \ref{prop_Oh}). In Section \ref{sec_multipliers}, we study multiplier functions and their analytic properties. In section \ref{sec_therm}, we review thermodynamic formalism and the pressure metric on the moduli space. We introduce a non-negative metric in Section \ref{sec_twoform} and prove the main result Theorem \ref{main} in Section \ref{sec_Riem}. 
As applications, we prove Theorems \ref{thm_introapp1} and Theorem \ref{thm_introapp2} in Section \ref{sec_nolocmax}.

\subsection{Acknowledgements} We would like to thank Martin Bridgeman, Genadi Levin and Kevin Pilgrim for useful conversations. We also thank Fei Yang and Michel Zinsmeister for references \cite{Przytycki06} and \cite{Ransford93}, respectively.

\section{Complex dynamics background} \label{sec_mul}
In this section, we give an expository account for the basics in complex dynamics. We prove three results for later use regarding the Hausdorff dimension of Julia sets (Proposition \ref{prop_HDJ}), the invariant measures (Proposition \ref{density}) and some property of multipliers (Proposition \ref{prop_Oh}).

\subsection{Hausdorff dimension of Julia sets for hyperbolic maps}
Let $f\in\mathbb{C}(z)$ be a rational map of degree at least $2$. Denote by $F(f)$ and $J(f)$ the Fatou set and Julia set of $f$, respectively. Recall that $f$ is \emph{hyperbolic} if all the critical points under iterations converge to attracting cycles. In this subsection, we state some results about the Hausdorff dimension of $J(f)$.

The following result, due to Przytycki \cite{Przytycki06}, states the dimensions of boundaries of immediate attracting basins. Recall that for an $f$-invariant set $K$, its hyperbolic Hausdorff dimension is the supremum of the Hausdorff dimensions of $f$-invariant subsets $X$ of $K$ such that $f|_{X}$ is expending. 
\begin{prop}\cite[Theorem A]{Przytycki06}\label{prop:boundary-hd}
Let $f\in\mathbb{C}(z)$ be a rational map of degree at least $2$. Suppose $f$ is not a finite Blaschke product in some holomorphic coordinates or a quotient of a Blaschke product by a rational function of degree $2$. Assume $f$ has an attracting cycle and denote $B$ its the immediate basin. If each component of $B$ is simply connected, then the hyperbolic Hausdorff dimension of $\partial B$ is larger than $1$.
\end{prop}

If $P\in\mathbb{C}[z]$ is a polynomial of degree at least $2$, the Julia set $J(P)$ is the boundary of the basin $B_\infty(P)$ of $\infty$. 
If $J(P)$ is connected, equivalently all the critical points of $P$ are away from $B_\infty(P)$, Zdunik's result \cite{Zdunik90} 
implies that the Hausdorff dimension of $J(P)$ is larger than $1$ unless $P$ is conjugate to a monomial or Chebyshev polynomial. Indeed,  in this case the Hausdorff dimension of the measure of maximal entropy for $P$ is $1$ \cite{Makarov85}.

For a hyperbolic polynomial $P$, if $J(P)$ is not a Cantor set, equivalently not all critical points are contained in $B_\infty(P)$, then the Hausdorff dimension of $J(P)$ larger than $1$ unless $P$ is conjugate to a monomial:

\begin{prop} \label{prop_HDJ}
Let $P$ be a hyperbolic polynomial of degree at least $2$. Suppose $P$ is not conjugate to a monomial. If $J(P)$ is not a Cantor set, then the Hausdorff dimension of $J(P)$ is larger than $1$.
\end{prop}
\begin{proof}
Since $P$ is hyperbolic and has a critical point not in the basin of $\infty$, it follows that $P$ has an attracting cycle in $\mathbb{C}$. Then the immediate basin of the attracting cycle is a union of finitely many simply connected components. Note the boundary of this immediate basin is contained in $J(P)$. Thus by Proposition \ref{prop:boundary-hd}, we only need to deal with the case that $P$ is a quotient of a Blaschke product by a rational function of degree $2$. In this case, we claim that $P$ has degree $2$. Indeed, for otherwise, the product of $P$ and a rational maps of degree $2$ is a Blaschke product. It is impossible. Since the Julia set of quadratic polynomial is either connected or a Cantor set, $J(P)$ is connected. Since $P$ is not conjugate to a monomial, it follows that the Hausdorff dimension of $J(P)$ is larger than $1$.
\end{proof}

Recall that $\mathrm{poly}_d$ is the moduli space of degree $d$ polynomials. Since the Hausdorff dimension of Julia sets are invariant under M\"obius conjugacy, the function $\delta:\mathrm{poly}_d\to(0,2)$, sending $[P]$ to the Hausdorff dimension of $J(P)$, is well-defined. The above proposition immediately implies the following property of $\delta$.

\begin{cor}
For $d\ge 2$, let $\mathcal{H}\subset\mathrm{poly}_d$ be a hyperbolic component. Suppose $\mathcal{H}$ is not in the shift locus. Then $\delta(\mathcal{H}-\{[z^d]\})\subset(1,2)$.
\end{cor}

\subsection{Primitive orbit measures}
Let $f\in\mathbb{C}(z)$ be a hyperbolic rational map of degree $d\ge 2$. A {\it primitive orbit measure} $\mu$ associated to $f$ is a probability measure uniformly supported on a periodic cycle $\hat x:=\{x, f(x),\cdots,f^{k-1}(x)\}$ in $J(f)$. Then for any continuous function $\phi$ defined on $J(f)$, we have  
$$\int_{J(f)} \phi d\mu = \frac{1}{k}\sum_{i=0}^{k-1}\phi(f^i(x)).$$
Denote by $\Omega(J(f))$ the space of $f$-invariant probability measures on $J(f)$. The next proposition states that the set of primitive orbit measures is dense in $\Omega(J(f))$ with respect to the weak-$*$ topology. This is because the dynamical system $(J(f),f)$ satisfies the {\it specification property}. Recall that for a compact metric space $X$ and a continuous transformation $T: X \to X$, we say that the topological dynamical system $(X,T)$ satisfies the \emph{specification property} if for any $\epsilon>0$, there exists an integer $N(\epsilon)$ such that for any $x_1,x_2\in X$ and two sets of consecutive positive integers $A_1=\{a_1,a_1+1,\cdots,b_1\}$ and $A_2=\{a_2,a_2+1,\cdots,b_2\}$ with $a_2-b_1>N(\epsilon)$, and for any integer $q>b_2-a_1+N(\epsilon)$, there exists a $q$-periodic point $x\in X$ of $T$ such that $\rho(T^jx,T^jx_1)<\epsilon$ for $j\in A_1$ and $\rho(T^kx,T^kx_1)<\epsilon$ for $k\in A_2$. See \cite{Sigmund74} for details.

\begin{prop}\label{density}
	If $f$ is a hyperbolic rational map of degree $d\ge 2$, the set of primitive orbit measures is dense in $\Omega(J(f))$.
\end{prop}
\begin{proof}
If $f\in\mathbb{C}(z)$ is hyperbolic, then $f$ is expanding in a neighborhood of $J(f)$ \cite[Theorem 3.13]{McMullen94}. Hence the dynamical system $(J(f),f)$ is semi-conjugate to a mixing subshift of finite type \cite{Ruelle89}. It follows that $(J(f),f)$ satisfies the specification property \cite[Page 287 and Proposition 1]{Sigmund74}. By \cite[Theorem 1]{Sigmund74}, the set of primitive orbit measures is dense in $\Omega(J(P))$.
\end{proof}

\subsection{Distribution of multipliers}
Let $f : \widehat{\mathbb{C}}\to\widehat{\mathbb{C}}$ be a hyperbolic rational map. Given any primitive periodic orbit $\hat{x} = \{x, f(x), \cdots, f^{n-1}(x) \}$ on the Julia set $J(f)$ of period $n$, we consider its  {\it multiplier} $\lambda(\hat{x})$ given by, in the local coordinates, 
$$\lambda(\hat{x}) = (f^n)'(x).$$
Let $\mathcal{O}$ denote the set of all primitive periodic orbits of $f$ in $J(f)$. For $T>0$, consider the counting function $$N_T(\mathcal{O}) \defeq  \# \{\hat{x} \in \mathcal{O} ~|~ |\lambda(\hat{x})|<T  \}.$$
Since $f$ is hyperbolic, $N_T(\mathcal{O})$ is finite for any $T > 0$.
In \cite{Oh17}, Oh and Winter proved the following asymptotics for $N_T(\mathcal{O}) $.

\begin{thm}\cite[Theorem 1.1 (1)]{Oh17}
	Let $f:\widehat{\mathbb{C}}\to\widehat{\mathbb{C}}$ be a hyperbolic rational map of degree at least $2$. Suppose that $f$ is not monomial. Then there exists $\epsilon>0$ such that 
	$$N_T(\mathcal{O})=Li(T^\delta)+O(T^{\delta-\epsilon})$$
	where $Li(t) = \displaystyle\int_{2}^{t} \frac{dt}{\log t}$ is the offset logarithmic integral and $\delta$ is the Hausdorff dimension of $J(f)$.
\end{thm}

Using the above theorem, we prove a result regarding the existence of multipliers within an annulus which we will use in Section \ref{sec_Riem}.
\begin{prop}\label{prop_Oh}
	Under the assumptions in the above theorem, we further assume that $\delta>1$. Let $\{T_n\}\subset\mathbb{R}_{>0}$ be a sequence with $T_n\to\infty$. Then for $S_n\ge T_n^{\alpha}$ with $\alpha>1-\delta$, 
	$$N_{T_n+S_n}(\mathcal{O})-N_{T_n}(\mathcal{O})\to\infty.$$
\end{prop}
\begin{proof}
It is sufficient to prove the conclusion for $\alpha<1$ and $T_n^{\alpha}\le S_n=o(T_n)$. Note that $Li(x)=x/\ln(x)+o(x/\ln(x))$ for sufficiently large $x>0$. There exists a constant $C_1\not=0$ such that 
	$$Li((T_n+S_n)^\delta)-Li(t_n^\delta)=C_1\frac{T_n^{\delta-1}S_n}{\ln T_n}+o\left(\frac{T_n^{\delta-1}S_n}{\ln T_n}\right).$$
	Let $\epsilon>0$ be as in the above theorem.
	Then there exists $C_2\ge 0$ such that 
	\begin{align*}
	&(N_{T_n+S_n}(\mathcal{O})-Li((T_n+S_n)^\delta))-(N_{T_n}(\mathcal{O})-Li(T_n^{\delta}))\\
	&=C_2T_n^{\delta-\epsilon-1}S_n+o(T_n^{\delta-\epsilon-1}S_n).
	\end{align*}
	It follows that 
	$$N_{T_n+S_n}(\mathcal{O})-N_{T_n}(\mathcal{O})=C_1\frac{T_n^{\delta-1}S_n}{\ln T_n}+o\left(\frac{T_n^{\delta-1}S_n}{\ln T_n}\right).$$
	Since $S_n\ge T_n^{\alpha}$ with $\alpha> 1-\delta$,  
	we have $N_{T_n+S_n}(\mathcal{O})-N_{T_n}(\mathcal{O})\to\infty.$ 
\end{proof}

For $I=(-1, 1)$, we say a family $\{f_t\}_{t\in I}$ is an \emph{analytic family of degree $d\ge 2$ rational maps} if each $f_t$ is a rational map of degree $d$ and each coefficient of $f_t$ is an analytic function of $t$. Compare to the holomorphic family of degree $d\ge 2$ rational maps, see \cite[Section 1.2]{Kiwi15}. Note that if $\{f_t\}_{t\in I}$ is an analytic family of degree $d\ge 2$ rational maps, the periodic points and hence the corresponding multipliers of $f_t$ are algebraic functions in $t$.

\begin{cor}\label{cor_multiplier}
	Let $\{f_t\}_{t\in I}$ be an analytic family of hyperbolic rational maps of degree at least $2$. Suppose that the Hausdorff dimension of $J(f_t)$ is larger than $1$ for all $t\in I$. Let $a_t$ be the multiplier of a repelling cycle of $f_t$ with $a'(0)\not=0$. Then for any $\kappa\in(0, 1)$, the following holds:
	for any arbitrary multiplier $b_t$ of a repelling cycle of $f_t$ and any sufficiently large positive integer $n\ge 1$, there exist $\theta_{t, n}:=\theta(t,n,\kappa, b_t)\in [0, 2\pi)$ and multipliers $\lambda_{t, n}$ of $f_t$ of the form  
	$$\lambda_{t, n }=e^{i\theta_{t, n}}(a_t^{n}+a_t^{n\kappa}b_t+o(a_t^{n\kappa }b_t)).$$
\end{cor}
\begin{proof}
	Pick $\varepsilon' \in (0,1)$.
	Consider the annulus
	$$A _{t,n}= \{ z \in \bbC ~|~ |a_t^n+ a_t^{n\kappa} b_t| \le |z| \le |a_t^n+ a_t^{n\kappa} b_t| + |a_t^{n\kappa} b_t|^{\varepsilon'}\}.$$
	We first claim that there exists a multiplier of $f_t$ in the annulus $A _{t,n}$ for sufficiently large $n$. Indeed,
	apply Proposition \ref{prop_Oh} for each $f_t$ with $T_n = |a_t^n+ a_t^{n\kappa} b_t| $ and $S_n = |a_t^{n\kappa} b_t|^{\varepsilon'}$. 
	
	Let $\lambda_{t, n}$ be such a multiplier of $f_t$ contained in $A_{t,n}$. Then, we must have $|\lambda_{t, n}| = |a_t^{n}+a_t^{n\kappa}b_t+o(a_t^{n\kappa}b_t)|$. Therefore $\lambda_{t, n }=e^{i\theta_{t, n}}(a_t^{n}+a_t^{n\kappa}b_t+o(a_t^{n\kappa }b_t))$ for some $\theta_{t, n}\in [0, 2\pi)$.
\end{proof}

\begin{rmk}
	Since $f_t$ has only countably many multipliers, there are uncountably many $\kappa\in(0,1)$ giving rise to the same values of $\lambda_{t,n}$ although the expressions of $\lambda_{t,n}$ are different.
\end{rmk}

\section{Multiplier functions} \label{sec_multipliers}
In this section, we introduce multiplier functions which we will need to construct our Riemannian metric. In particular, we give the definitions in Subsection 3.1 and discuss the analytic properties in Subsection 3.2. In Subsection 3.3, we consider complex multiplier functions which are natural extensions of the (real) multiplier functions.
\subsection{Definitions}
Let $f$ be a hyperbolic rational map of degree at least $2$ and let $\mu \in \Omega(J(f))$ be a primitive orbit measure. Denote by $\hat x=\{x,f(x),\cdots,f^{k-1}(x)\}$ the support of $\mu$ and $\lambda(\hat x)$ the multiplier of $\hat{x}$. Define
$$m_{f}(\mu):= \frac{1}{k}\log | \lambda(\hat x)|.$$

\begin{lem}
	There exists a unique continuous function $M_{f} : \Omega(J(f)) \to \bbR$ such that $M_f(\mu)=m_f(\mu)$ if $\mu\in \Omega(J(f))$ is a primitive orbit measure.
\end{lem}
\begin{proof}
	Given $\mu\in\Omega(J(f))$, we define
	$$M_f(\mu) = \int_{J(f)} \log|f'| d\mu.$$ 
	Then if $\mu$ is a primitive orbit measure, by definition of  $m_f(\mu)$, we have $M_f(\mu)=m_f(\mu)$. The continuity of $M_f$ follows immediately from the weak-$\ast$ topology on $\Omega(J(f))$ since $\log|f'|$ is continuous on $J(f)$. The uniqueness of $M_f$ follows from the density of primitive orbit measures in $\Omega(J(f))$. Indeed, if $\widehat{M}_f$ is another continuous function on $\Omega(J(f))$ such that $\widehat{M}_f=m_f$ on the subset of primitive orbit measures, then by Proposition \ref{density}, we have $\widehat{M}_f=M_f$ on a dense subset of $\Omega(J(f))$. It follows that $\widehat{M}_f=M_f$ on $\Omega(J(f))$.
\end{proof}

Let $\widetilde{\mathcal{H}}$ be a hyperbolic component in $\mathrm{Rat}_d$ and let $U(f)\subset \widetilde{\mathcal{H}}$ be defined in Section \ref{intro}. For $g\in U(f)$, recall that $\phi_g : J(f) \to J(g)$ is the corresponding homeomorphism conjugating the dynamics $f : J(f) \to J(f)$ to $g : J(g) \to J(g)$.


\begin{lem}
	The pushforward map $(\phi_g)_* : \Omega(J(f)) \to \Omega(J(g))$ is a homeomorphism.
\end{lem}
\begin{proof}
	Since $(\phi_g)_*$ is functorial, it is one-to-one and onto. If $\mu_n \to \mu$ in $\Omega(J(f))$ and $h$ is a continuous function on $J(g)$, then $((\phi_g)_*\mu_n)(h) = \mu_n(h \circ \phi_g)$. Note that $h \circ \phi_g$ has compact support since $\phi_g$ is a homeomorphism. As $\mu_n \to \mu$ in $\Omega(J(f))$, we have $\mu_n(h \circ \phi_g) \to \mu(h \circ \phi_g)$, i.e. $(\phi_g)_*\mu_n \to (\phi_g)_*\mu$. Hence, $(\phi_g)_*$ is continuous. Similarly, the map $(\phi_g^{-1})_*$ is continuous. Moreover, the composition of $(\phi_g)_*$ and $(\phi_g^{-1})_*$ is the identity.
\end{proof}

We obtain the following corollary as an immediate consequence of the two previous lemmas.
\begin{cor}\label{cor_cts}
	The map $M_{g} \circ (\phi_g)_* : \Omega(J(f)) \to \bbR$ is continuous.
\end{cor}

For the $f \in \widetilde{\mathcal{H}}$, we define a map $M : \Omega(J(f)) \times U(f) \to \bbR$ by
\begin{equation*}
M(\mu, g) = M_{g} \left((\phi_g)_* \mu \right).
\end{equation*}
Then for any fixed $g\in U(f) $, by Corollary \ref{cor_cts}, the function $M$ is continuous in $\mu$. We will discuss analytic properties of $M$ in the next subsection.

\begin{defn}
Given any $\mu \in \Omega(J(f))$, we call $M_{\mu}(g):= M(\mu, g)$ the {\it multiplier function} associated to $\mu$.
\end{defn}

\begin{rmk}
	If $\mu \in \Omega(J(f))$ is a primitive orbit measure, then it is supported on a primitive periodic orbit $\hat x$ of period $k$ for some $k\ge 1$ in $J(f)$ and 
	$$M_{\mu}(g) =\frac{1}{k}\log \big| \lambda(\phi_g(\hat x))\big|.$$
\end{rmk}

\subsection{Analyticity of multiplier functions} \label{sec_anamul}
In this subsection, we discuss analytic properties of the map $M : \Omega(J(f)) \times U(f) \to \bbR$ as defined in the previous subsection. 

We recall that if $X$ is a smooth manifold and $C^{\infty}(X,\bbR)$ is the set of smooth real-valued functions on $X$, the $C^{\infty}$-topology on $C^{\infty}(X,\bbR)$ is given by $\psi_n \to \psi$ if the derivatives of $\psi_n$ converge uniformly on compact subsets of $X$ to the derivatives of $\psi$. Our main goal  in this subsection is to prove the following result.
\begin{prop} \label{thm_Manalytic}
The map $M : \Omega(J(f)) \times U(f)\to \bbR$ satisfies the following properties:
\begin{enumerate}
	\item for each $\mu \in \Omega(J(f))$, the multiplier function $M_{\mu} : U(f) \to \bbR$ is real analytic; and
	\item the map from $\Omega(J(f))$ to $C^{\infty}(U(f), \bbR)$, sending $\mu$ to $M_{\mu}$, is continuous.
\end{enumerate}
\end{prop}

To prove the proposition, we first show that if $\mu_n \to \mu$ in $\Omega(J(f))$, the multiplier functions $M_{\mu_n}$ converge uniformly on compact subsets of $U(f)$. 

\begin{lem} \label{lem_unif}
	If $\mu_n \to \mu$ in $\Omega(J(f))$, then the multiplier functions $M_{\mu_n} \to M_{\mu}$ uniformly on compact subsets of $U(f)$.
\end{lem}
\begin{proof}
	Let $g \in U(f)$. For any $\epsilon>0$, we first claim that there is a neighborhood $W$ of $g$ in $U(f)$ such that for any $\nu\in\Omega(J(f))$ and any $h\in W$, 
	$$|M_{\nu}(g)-M_{\nu}(h)|<\epsilon.$$
	Indeed, 
	\begin{align*}
	|M_{\nu}(g)-M_{\nu}(h)|&=\left|\int_{J(f)} \log|g'\circ\phi_g| d\nu-\int_{J(f)} \log| h'\circ\phi_{h}| d\nu\right|\\
	&=\left|\int_{J(f)} (\log|g'\circ\phi_g|-\log|h'\circ\phi_{h}|) d\nu\right|\\
	&=\left|\int_{J(f)} \log\left|\frac{g'\circ\phi_g}{h'\circ\phi_{h}} \right| d\nu\right|\\
	&\le \max\left\{\left|\log\left|\frac{g'\circ\phi_g(z)}{h'\circ\phi_{h}(z)} \right|\right|:z\in J(f)\right\}.
	\end{align*}
	Consider
	$$\alpha_g(h,z):=\frac{g'\circ\phi_g(z)}{h'\circ\phi_{h}(z)}.$$ 
	Then $\alpha_g$ is well-defined on $U(f)\times J(f)$ since the Julia sets do not contain critical points. Moreover, $\alpha_g$ is continuous in both $h$ and $z$.
	Note that $\alpha_g(g,z)=1$ for all $z\in J(f)$. Since $J(f)$ is compact, we can choose a sufficiently small neighborhood $W$ of $g$ such that  $|\alpha_g(h,z)|<e^\epsilon$ for all $h\in W$ and all $z \in J(f)$. Hence the claim holds.
	
	Now consider the sequence $\{\mu_n\}$. For any $\epsilon>0$, by the previous claim, we can choose a neighborhood $V$ of $g$ in $U(f)$ such that for all $\mu_n$ and all $h\in V$, we have $|M_{\mu_n}(g)-M_{\mu_n}(h)|<\epsilon$. It follows that the sequence $\{M_{\mu_n}\}$ is equicontinuous on any compact subset of $U(f)$. Moreover, by the definition of $M_{\mu_n}$ and $M_{\mu}$, we have that $M_{\mu_n}$ converges pointwise to $M_{\mu}$. It follows that $M_{\mu_n}$ locally uniformly converges to $M_{\mu}$.
\end{proof}

\begin{proof}[Proof of Proposition \ref{thm_Manalytic}]
	Let $\mu \in \Omega(J(f))$. By Proposition \ref{density}, there exists a sequence $\mu_n$ of primitive orbit measures  in $\Omega(J(f))$ such that $\mu_n \to \mu$. We note that $M_{\mu_n}$ is harmonic. Indeed, note that  
	$$M_{\mu_n}(g) =  \int_{J(f)} \log|g'\circ\phi_g| d\mu_n.$$
	Since $\phi_g$ is analytic in $g$, the map $g\mapsto g'\circ\phi_g$ is analytic in $g$. Then $\log|g'\circ\phi_g|$ is harmonic in $g$. By Lemma \ref{lem_unif}, the sequence $M_{\mu_n}$ converges to $M_{\mu}$ uniformly on compact sets. Therefore $M_{\mu}$ is harmonic. In particular, it is real-analytic. This completes the proof of statement $(1)$.
	
For statement $(2)$, if $\mu_n \to \mu$, again by Lemma \ref{lem_unif}, the sequence $M_{\mu_n}$ converges to $M_{\mu}$ unifromly on compact sets. As $M_{\mu_n}$ are harmonic, uniform convergence on compact sets implies uniform convergence of derivatives on compact sets.
\end{proof}

\subsection{Complex multiplier functions} \label{sec_cxmul}
In this subsection, we introduce complex multiplier functions which are natural extensions of the real multiplier functions.  

Consider the following map $\widetilde{M}: \Omega(J(f)) \times U(f)\to \bbC$ defined by
$$\widetilde{M}(\mu, g) = \int_{J(g)} \log g' ~d((\phi_g)_*\mu) =  \int_{J(f)} \log (g'\circ\phi_g) d\mu.$$
For each $\mu\in\Omega(J(f))$, let $\widetilde{M}_{\mu}: U(f)\to \bbC$ be the function $\widetilde{M}_{\mu}(g):=\widetilde{M}(\mu, g)$.
Then the multiplier function $M_{\mu}$ defined in previous subsection equals the real part of $\widetilde{M}_{\mu}$. We call $\widetilde{M}_{\mu}$ the {\it complex multiplier function} associated to $\mu$.

The space $C^{\omega}(U(f), \bbC)$ consists of all holomorphic functions on $U(f)$ and we equip it with the topology of uniform convergence on compact subsets. Parallel to Proposition \ref{thm_Manalytic}, we have the following result for $\widetilde{M}_{\mu}$. The proof of Proposition \ref{thm_Manalytic} works verbatim here.
\begin{prop}
	The map $\widetilde{M}: \Omega(J(f)) \times U(f) \to \bbC$ satisfies the following properties:
\begin{enumerate}
\item for each $\mu \in \Omega(J(f))$, the complex multiplier function $\widetilde{M}_{\mu}: U(f) \to \bbC$ is holomorphic; and
\item the map from $\Omega(J(f))$ to $C^{\omega}(U(f), \bbC)$, sending $\mu$ to $\widetilde{M}_{\mu}$, is continuous.
\end{enumerate}	
\end{prop}

\section{Thermodynamic formalism and the pressure metric} \label{sec_therm}
In this section, we first review the Thermodynamic Formalism for conformal repellers. In particular, we discuss the topological pressure of a H\"older continuous function and the pressure metric on the space of cohomology classes of H\"older continuous functions with pressure zero. Standard references are \cite{Parry90, Ruelle04, Zinsmeister00}. Via the thermodynamic embedding map, the pressure metric pulls back to a non-negative two-form on a hyperbolic component in the moduli space $\mathrm{rat}_d$ of degree $d\ge 2$ rational maps.

\subsection{Conformal repellers}\label{repeller}
Let $f$ be a holomorphic function from an open subset $V\subset\bbC$ into $\bbC$ and let $J$ be a compact subset of $V$. The triple $(J,V,f)$ is a {\it conformal repeller} if
\begin{enumerate}
	\item there exist $C>0$ and $\alpha>1$ such that $|(f^n)'(z)| \ge C\alpha^n$ for every $z \in J$ and $n \ge 1$,
	\item $f^{-1}(V) \subset V$ is relatively compact in $V$ with $J = \cap_{n \ge 1}f^{-n}(V)$, and
	\item for any open set $U$ with $U \cap J \neq \emptyset$, there exists an $n>0$ such that $J \subset f^n(U \cap J)$.
\end{enumerate}

An important property of conformal repellers is the existence of a Markov partition. A {\it Markov partition} of $J$ is a finite cover of $J$ by sets $R_j$, $1 \le j \le N$ satisfying the following conditions:
\begin{enumerate}
	\item each set $R_j$ is the closure of its interior $\text{Int}R_j$,
	\item the interiors of the $R_j$ are pariwise disjoint,
	\item if $x \in \text{Int}R_j$ and $f(x)\in \text{Int}R_\ell$, then $R_\ell \subset f(R_j)$, and
	\item each restriction $f|_{R_j}$ is injective.
\end{enumerate}

Let $(J,V,f)$ be a conformal repeller and let $(R_1, \cdots, R_m)$ be a Markov partition of $J$. Define a matrix $A$ by $$
A_{j, \ell}=
\begin{cases}
1, &\mbox{if }  R_\ell \subset f(R_j),\\
0, & \mbox{otherwise}.
\end{cases}.
$$
Then every point $x \in J$ corresponds to an infinite sequence $\{\ell _k\}_{k\ge 0}$ where $\ell _k \in \{1, \cdots, m \}$ and $A_{\ell_k,\ell_{k+1}} = 1$. Let $\Sigma$ be the set of all such sequences, i.e.
$$\Sigma = \{ \{\ell _k\}_{k \ge 0} ~|~  \ell _k \in \{1, \cdots, m \}, A_{\ell_k,\ell_{k+1}} = 1\} $$
and $\sigma : \Sigma \to \Sigma$ be the shift map, i.e. $$\sigma(\ell _0, \ell _1, \ell _2\cdots) = (\ell _1, \ell _2, \ell _3, \cdots).$$ Then there is a projection map $\Psi_f : \Sigma \to J$ sending a sequence $\{\ell _k\}_{k \ge 0}$ to $x \in J$ such that $f^{\ell_k}(x) \in R_{\ell _k}$.

Denote $C^{\alpha}(\Sigma)$ the space of $\alpha$-H\"older continuous real-valued functions on $\Sigma$. We say that two functions $\phi_1, \phi_2 \in C^{\alpha}(\Sigma)$ are {\it cohomologous}, denoted by $\phi_1 \sim \phi_2$, if there exists a continuous function $h : \Sigma \to \bbR$ such that $\phi_1(x) - \phi_2(x) = h(\sigma(x)) - h(x)$. In particular, we say a function $\phi \in C^{\alpha}(\Sigma)$ is  a {\it coboundary} if $\phi \sim 0$.

Let $\phi_f = - \log|f' \circ \Psi_f|$. Then $\phi_f$ is a H\"older continuous function. Bowen's theorem states that the Hausdorff dimension of $J$ is the unique solution to the equation $\mathcal{P}(t\phi_f) = 0$, see \cite{Zinsmeister00}. Here $\mathcal{P}$ is the topological pressure which we introduce now.

\subsection{The pressure function}
In this subsection, we review definitions of topological pressure and equilibrium states. Then we summarize formulas for the derivatives of the pressure function. A standard reference is \cite{Parry90}.

Given $\phi \in C^{\alpha}(\Sigma)$, the {\it transfer operator} $\mathcal{L}_{\phi} : C^{\alpha}(\Sigma) \to C^{\alpha}(\Sigma)$ is defined by 
$$\mathcal{L}_{\phi}(g)(y) = \sum_{f(x) = y} e^{\phi(x)}g(x).$$

By Ruelle-Perron-Frobenius theorem, there is a positive eigenfunction $e^{\psi}$, unique up to scale, such that $$\mathcal{L}_{\phi}(e^{\psi}) = \rho(\mathcal{L}_{\phi}) e^{\psi},$$
where $\rho(\mathcal{L}_{\phi})$ is the isolated maximal eigenvalue of the transfer operator and the rest of the spectrum is contained in a disk of radius $r < \rho(\mathcal{L}_{\phi})$.

The {\it pressure} of $\phi$ is defined by $$\mathcal{P}(\phi) = \log \rho(\mathcal{L}_{\phi}).$$

Alternatively, the pressure $\mathcal{P}(\phi)$ can also be defined using variational methods. Let $\Omega^{\sigma}$ be the set of $\sigma$-invariant probability measures on $\Sigma$. Then
$$\mathcal{P}(\phi)  = \sup_{m \in\Omega^{\sigma}} \left(h(\sigma, m) + \int_{\Sigma} \phi dm \right).$$
where $h(\sigma,m)$ is the measure-theoretic entropy of $\sigma$ with respect to the measure $m$. A measure $m=m(\phi) \in\Omega^{\sigma}$ is called an {\it equilibrium state} of $\phi$ if $\mathcal{P}(\phi)  =h(\sigma, m) + \int_{\Sigma} \phi dm$.

The equilibrium state $m(\phi)$ is also related to the spectral data of transfer operators. If $\mathcal{P}(\phi) = 0$, then $\mathcal{L}_{\phi}(e^{\psi}) = e^{\psi}$. It follows that there is a unique positive measure $\mu$ on $\Sigma$ such that $$\int_\Sigma \mathcal{L}_{\phi}(\tilde\phi) d\mu = \int_\Sigma\tilde\phi d\mu$$ for all $\tilde\phi\in C^{\alpha}(\Sigma)$ and $\int_\Sigma e^{\psi} d\mu =1$. We have $$m(\phi) = e^{\psi}\mu.$$ Note that $m(\phi)$ an ergodic, $\sigma$-invariant probability measure with positive entropy.

We summarize the following formulas for the derivatives of the pressure $\mathcal{P}$.
\begin{prop}\cite[Propositions 4.10 and 4.11]{Parry90}
If $\phi_t$ is a smooth path in $C^{\alpha}(\Sigma)$, we have
$$\frac{d\mathcal{P}(\phi_t)}{dt} \bigg|_{t=0} = \int_{\Sigma} \dot{\phi_0} dm,$$
 where $m = m(\phi_0)$ and $\dot{\phi_0} = d\phi_t/dt |_{t=0}$.
If the above first derivative of $\mathcal{P}(\phi_t)$ is zero, then
$$\frac{d^2\mathcal{P}(\phi_t)}{dt^2} \bigg|_{t=0} = Var(\dot{\phi_0}, m), $$
where  $Var(\dot{\phi_0}, m)$ is the variance of $\dot{\phi_0}$ with respect to $m$.
\end{prop}

\subsection{The pressure metric}
The pressure function $\mathcal{P} : C^{\alpha}(\Sigma) \to \bbR$ is convex, real-analytic and depends only on the cohomology classes. We show in this subsection that it defines a metric in the thermodynamic setting.

Let $\mathcal{C}(\Sigma)$ be the set of cohomology classes of H\"older continuous functions with pressure zero, that is,
$$\mathcal{C}(\Sigma) = \{\phi: \phi \in C^{\alpha}(\Sigma), \mathcal{P}(\phi) = 0 \}/ \sim$$
where $\phi_1 \sim \phi_2$ if $\phi_1$ and  $\phi_2$ are cohomologous.

If $[\phi] \in\mathcal{C}(\Sigma)$, let $m$ is an equilibrium state for $\phi$. Then by the formula for the derivative of the pressure $\mathcal{P}$, the tangent space of $\mathcal{C}(\Sigma)$ at $[\phi]$ can be identified with 
$$T_{[\phi]}\mathcal{C}(\Sigma) = \left\{\psi ~\Big|~ \int_{\Sigma} \psi dm = 0  \right\} / \sim.$$

By convexity of $\mathcal{P}$, the second derivative
$$\frac{d^2\mathcal{P}(\phi + t \psi)}{dt^2} \bigg|_{t=0} = Var(\psi, m(\phi))$$ is non-negative on the tangent space 
$T_{[\phi]}\mathcal{C}(\Sigma)$. In fact, the variance is zero if and only if $\psi$ is cohomologous to zero (\cite{Parry90}, Proposition 4.12). Therefore, the {\it pressure metric} $||\cdot||_{pm}$ on $\mathcal{C}(\Sigma)$ given by $$||[\psi]||_{pm} = \frac{Var(\psi,m)}{-\int_{\Sigma} \phi dm}$$
is non-degenerate.

\subsection{Thermodynamic embedding of hyperbolic components}
Let $\mathcal{H}\subset\mathrm{rat}_d$ be a hyperbolic component. For $[f] \in \mathcal{H}$, there exists a neighborhood $V$ of $J(f)$ such that $(J(f),V,f)$ is a conformal repeller. Moreover, $(J(f), f)$ admits a Markov partition $R_1,R_2, \cdots, R_m$ for some $m \in \bbN$.
Recall that $\Psi_f: \Sigma \to J(f)$ is the projection map as in Subsection \ref{repeller}.

The function $-\log|f'(\Psi_f(\cdot))| : \Sigma \to \bbR$ is H\"older continuous and by Bowen's theorem, we have $$\mathcal{P}(-\delta(f)\log|f'(\Psi_f(\cdot))|) = 0, $$ where $\delta(f)$ is the Hausdorff dimension of the Julia set $J(f)$. Note that if $f_1\in\mathrm{Rat}_d$ is M\"obius conjugate to $f$, then $\delta(f_1)=\delta(f)$ and $\log|(f^n_1)'(\Psi_{f_1}(\cdot))|)=\log|(f^n)'(\Psi_f(\cdot))|)$ on the $n$-periodic points of $\sigma$ for all $n\ge 1$. It follows that on the $n$-periodic points of $\sigma$, $$-\delta(f_1)\log|(f^n_1)'(\Psi_{f_1}(\cdot))|=-\delta(f)\log|(f^n)'(\Psi_f(\cdot))|.$$ 
By Livsic Theorem, we have $-\delta(f_1)\log|f_1'(\Psi_{f_1}(\cdot))|$ and $-\delta(f)\log|f'(\Psi_f(\cdot))|$ are cohomologous. Thus, there is a {\it thermodynamic embedding} $$\mathscr{E} : \mathcal{H} \to\mathcal{C}(\Sigma),$$ given by
$$\mathscr{E}(f) = [-\delta(f)\log|f'(\Psi_f(\cdot))|].$$

We define a non-negative metric $||\cdot||_{\mathcal{P}}$ on $\mathcal{H}$ as the pullback of the pressure metric on $\mathcal{C}(\Sigma)$. Indeed, $||\cdot||_{\mathcal{P}}$ is non-negative since the pressure metric is positive-definite. Abusing notation, we also call $||\cdot||_{\mathcal{P}}$ the {\it pressure metric} on $\mathcal{H}$.

Now we derive a formula for $||\cdot||_{\mathcal{P}}$. Given $[f] \in \mathcal{H}$ and $v \in T_{[f]}\mathcal{H}$, let $c(t):=[f_t], t \in (-1, 1)$ be a path in $\mathcal{H}$ such that $c(0) = [f]$ and $c'(0) = v$. Under the thermodynamic embedding, $c(t)$ corresponds to the following  one-parameter family of pressure zero H\"older functions on $\Sigma$:
$$g(t,z) = -\delta(f_t)\log|f_t'(\Psi_{f_t}(\cdot))|, ~~t \in (-1, 1).$$ Denote by
$\dot{g}(0,z) = \frac{d}{dt}\big|_{t=0}g(t,z)$. Then by definition of the pressure metric,
$$||v||_{\mathcal{P}}^2  = \frac{Var(\dot{g}(0,z), \nu)}{-\int_{\Sigma}g(0,z) d\nu}$$
where $\nu$ is the equilibrium state for $g(0,z)$.

Since $\mathcal{P}(g(t,z)) = 0$ for all $t \in (-1, 1)$, by taking derivative with respect to $t$, we obtain $$\mathcal{P}'(g(t,z))\dot{g}(t,z) = 0.$$ Taking derivative with respect to $t$ again, we obtain $$\mathcal{P}''(g(t,z)) \dot{g}^2(t,z) + \mathcal{P}'(g(t,z)) \ddot{g}(t,z) = 0.$$ Evaluating at $t=0$, we further obtain $$Var(\dot{g}(0,z), \nu) + \int_{\Sigma} \ddot{g}(0,z) d\nu = 0.$$
Therefore, we have
\begin{align*} 
||v||_\mathcal{P}^2 & = \frac{Var(\dot{g}(0,z), \nu)}{-\int_{\Sigma}g(0,z) d\nu} = \frac{\int_{\Sigma} \ddot{g}(0,z) d\nu}{\int_{\Sigma}g(0,z) d\nu}.
\end{align*}

\section{A symmetric bilinear form $||\cdot||_G$} \label{sec_twoform}
Let $\widetilde{\mathcal{H}}\subset\mathrm{Rat}_d$ be a hyperbolic component. Our main goal in this section is to define a non-negative $2$-form $||\cdot||_G$ on $\widetilde{\mathcal{H}}$. This $2$-form descends to a non-negative $2$-form on a hyperbolic component in $\mathrm{rat}_d$ in the next section.  Moreover, we will also show that in the next section the descended $2$-form in fact is a Riemann metric in certain hyperbolic components in $\mathrm{poly}_d$. 

Fix $f\in \widetilde{\mathcal{H}}$. Let $U(f)$ be as in Section \ref{intro} and let $\nu$ be the equilibrium state of the H\"older potential $-\delta(f)\log |f'| : J(f) \to \bbR$ which has pressure zero.
Recall that the multiplier function $M_{\nu} : U(f)\to \bbR$ is given by
$$M_{\nu}(g) = \int_{J(f)} \log|g' \circ \phi_g|d\nu = \int_{J(g)} \log|g'| d\left((\phi_g)_* \nu\right),$$
where $\phi_g : J(f) \to J(g)$ is a quasi-conformal conjugacy. Now consider the real analytic function $G_{f}: U(f) \to \bbR$ given by $G_f(g)=\delta(g) M_{\nu}(g)$.
In what follows, we will show the Hessian of the function $G_{f}$ is well-defined at $f$ and gives us a non-negative $2$-form on $\widetilde{\mathcal{H}}$. Note that the Hessian of a smooth real-valued function $G : X \to \bbR$ on a smooth manifold $X$ is not well-defined at a point $x \in X$ unless $G'(x) = 0$ (see \cite[Section 7]{Bridgeman08}). We first show $G_f$ has a minimum at $f$ and hence $G_f'(f)=0$ in the following result.


\begin{prop}\label{prop_metric}
Fix $f \in\widetilde{\mathcal{H}}$. Then for all $g\in U(f)$, we have $$\frac{\delta(f)}{\delta(g)} \le \frac{M_{\nu}(g)}{M_{\nu}(f)}.$$
\end{prop}
\begin{rmk}
Note that the quantities $M_{\nu}(f)$ and $M_{\nu}(g)$ are respectively the Lyapunov exponents of $f$ and $g$ with respect to the equilibrium state $\nu$ and its pushforward $(\phi_{g})_\ast \nu$, respectively. Then the right-hand-side ratio measures the {\it distortion} of the Lyapunov exponents with respect to the equilibrium state under quasi-conformal deformations. Therefore, the proposition states that the {\it distortion ratio} of Lyapunov exponents is bounded below by the ratio of the Hausdorff dimensions of the Julia sets.
\end{rmk}

As an immediate consequence of the proposition, we see that a quasi-conformal deformation increases the Hausdorff dimension if it decreases the Lyapunov exponent with respect to the pushforward of the equilibrium state.
\begin{cor}
Fix notations as above. If $M_{\nu}(g) < M_{\nu}(f)$, then $\delta(f) < \delta(g)$.
\end{cor}

\begin{proof}[Proof of Proposition \ref{prop_metric}]
	Set $m_{g}:=(\phi_{g})_\ast \nu$. Then
	$$M_{\nu}(g) = \int_{J(f)} \log|g' \circ \phi_g|d\nu = \int_{J(g)} \log|g'|d m_{g}.$$
	
	Since $-\delta(f) \log|f'| : f \to \bbR$ has pressure zero and $\nu = m_{f}$ is its equilibrium state,
	by the variational definition of pressure,
	$$h(f, m_f) = \delta(f)\int_{J(f)} \log|f'| dm_{f},$$
	where $h(f, m_f)$ is the measure-theoretic entropy of $f$ with respect to $m_f$.
	
	Since entropy is invariant under topological conjugacy, it follows that
	$h(f, m_f)=h(g, (\phi_g)_\ast m_f)$.
	We have $h(f, m_f)=h(P, m_g)$.
	Since  $m_f$ is $f$-invariant, $(\phi_g)_\ast m_f$ is $f$-invariant and hence $m_g$ is $g$-invariant. 
	Again, by the variational definition of pressure, we have
	$$h(g, m_g) \le \delta(g) \int_{J(g)} \log|g'| d m_g.$$
	Hence $\delta(f)M_{\nu}(f)\le\delta(g)M_{\nu}(g)$ and the conclusion follows.
\end{proof}

Therefore, the Hessian of $G_f$ at $f$ is well-defined and it defines a symmetric bilinear form $||\cdot||_G$ on the tangent space $T_f\widetilde{\mathcal{H}}$ as follows. 
Let $\tilde{c}(t), t \in (-1, 1)$ be a path in $U(f)$ with $\tilde{c}(0) = f$ and $\tilde{c}'(0)=\tilde{v} \in T_f\widetilde{\mathcal{H}}$.
Define 
$$||\tilde{v}||_G^2 := \frac{\partial^2 G_{f}}{\partial{\tilde{v}} \partial{\tilde{v}} } = \frac{d^2}{dt^2} \bigg|_{t=0} G_{f}(\tilde{c}(t)).$$
It is easy to check that $||\tilde{v}||_G^2$ only depends on $f$ and $\tilde{v}$. Moreover, by Proposition \ref{prop_metric}, we have that $||\tilde{v}||_G^2 \ge 0$.

\section{Proof of Theorem \ref{main}} \label{sec_Riem}
In this section, we first show that our symmetric bilinear form $||\cdot||_G$ in the previous section descends to a $2$-form on the hyperbolic components in $\mathrm{rat}_d$. Then we prove Theorem \ref{main}. For quadratic polynomials, we also show that $||\cdot||_G$ is positive-definite on the central component.

\subsection{The $2$-form on hyperbolic components in $\mathrm{rat}_d$}\label{sec:2-form}
Let  $\mathcal{H}\subset\mathrm{rat}_d$ be a hyperbolic component. 
For $[f]\in\mathcal{H}$ and $v\in T_{[f]}\mathcal{H}$, a curve $c(t)$ in $\mathcal{H}$ with $c(0)=[f]$ and $c'(0)=v$, consider two distinct lifts $\tilde{c}(t)$ and $\tilde{c}_1(t)$ in $\mathrm{Rat}_d$. Since our analysis is local, we may assume that $\tilde{c}(t)\subset U(\tilde{c}(0))$ and $\tilde{c}_1(t)\subset U(\tilde{c}_1(0))$ as in the previous section. By the definition of $||\cdot||_G$, we have $||\tilde{c}'(0)||_G=||\tilde{c}_1'(0)||_G$. Indeed, since  $\tilde{c}(t)$ and  $\tilde{c}_1(t)$ are M\"obius conjugate, $G_{\tilde{c}(0)}(\tilde{c}(t))=G_{\tilde{c}_1(0)}(\tilde{c}_1(t))$ on $(-1, 1)$. Thus the $2$-form $||\cdot||_G$ descends to a $2$-form on $\mathcal{H}$. Abusing notation, we also denote the $2$-form on $\mathcal{H}$ by $||\cdot||_G$ and therefore 
$$||v||_G:=||\tilde{c}'(0)||_G.$$

Write $\tilde c(t)=f_t\in\mathrm{Rat}_d$. For $z \in \Sigma$, recall from Section \ref{sec_therm} that 
$$g(t,z) = -\delta(f_t)\log|f_t'\circ \Psi_{f_t}(z)|,$$ 
where $t \in (-1, 1)$. Denote by $\dot{g}(0,z) = \frac{d}{dt}\big|_{t=0} g(t,z)$ and let $\nu$ be the equilibrium state for $g(0,z)$.

\begin{prop}\label{prop_conf}
	The form $||\cdot||_G$ is conformal equivalent to the pressure form $||\cdot||_{\mathcal{P}}$. More precisely, we have
	$$||v||_{\mathcal{P}}^2 = \frac{||v||_G^2}{\int_{\Sigma}g(0,z) d\nu}.$$
\end{prop}
\begin{proof}
	By straightforward calculation,
	$$||v||_\mathcal{P}^2 = \frac{Var(\dot{g}(0,z), \nu)}{-\int_{\Sigma}g(0,z) d\nu} = \frac{\int_{\Sigma} \ddot{g}(0,z) d\nu}{\int_{\Sigma}g(0,z) d\nu}= \frac{||v||_G^2}{\int_{\Sigma}g(0,z) d\nu}.$$
	The last equality holds by definition of $||\cdot||_G$.
\end{proof}
\begin{cor}\label{cor_G}
Fix the notations as above. Then $||v||_G = 0$ if and only if $||v||_{\mathcal{P}} =0$ if and only if $\dot{g}(0,z)$ is a coboundary.
\end{cor}
\begin{proof}
It immediately follows from Proposition \ref{prop_conf} and a standard fact from Thermodynamic Formalism that $Var(\dot{g}(0,z), \nu) = 0$ if and only if $\dot{g}(0,z)$ is a coboundary.
\end{proof}

\subsection{$||\cdot||_G$ on hyperbolic components in $\mathrm{poly}_d$}
Recall that $\mathcal{S}_d$ and $\mathcal{H}_0$ are the shift locus and the central hyperbolic component in $\mathrm{poly}_d$, respectively. Let $\mathcal{H}\subset\mathrm{poly}_d\setminus(\mathcal{S}_d\cup\mathcal{H}_0)$ be a hyperbolic component. In this subsection we show the following result which asserts that $||\cdot||_G$ is positive-definite on $\mathcal{H}$.

\begin{thm}
For any $[P]\in \mathcal{H}$ and for any nonzero $v \in T_{[P]}\mathcal{H}$, we have $||v||_{G}>0$.
\end{thm}

\begin{proof}
We prove by contradiction. Consider a path $c(t)=[P_t]$ with $P_0=P$ and $c'(0)=v$. Suppose $||v||_G= 0$ and $v \neq 0$. By Corollary \ref{cor_G}, the map $\dot{g}(0,z)= -\delta(P_t)\log|P_t'\circ \Psi_{P_t}(z)|$ is a coboundary. 

Let $\sigma:\Sigma\to\Sigma$ be the map such that $P\circ\Psi_P=\Psi_P\circ \sigma$. By definition of coboundary, there exists a continuous function $h: \Sigma\to\mathbb{R}$ such that $\dot{g}(0,z)=h(z)-h(\sigma(z))$. Let $z\in\Sigma$ be a periodic point of $\sigma$, i.e. $\sigma^n(z)=z$ for some $n\ge 1$. Then 
\begin{align*} 
0&= h(z) - h(\sigma^n(z))\\
&= \frac{d}{dt}\bigg|_{t=0} g(t,z)+\frac{d}{dt}\bigg|_{t=0} g(t,\sigma(z))+\cdots+\frac{d}{dt}\bigg|_{t=0} g(t,\sigma^{n-1}(z))\\
&=-\frac{d}{dt}\bigg|_{t=0} \delta(P_t)\log|(P_t^n)'\circ \Psi_{P_t}(z)|.
\end{align*} 
                                                             
Applying the chain rule, we obtain
\begin{align*}
\frac{d}{dt}\Big|_{t=0} \log |(P^n_t)'\circ \Psi_{P_t}(z)| &= -\frac{\frac{d}{dt}\big|_{t=0} \delta(P_t)}{\delta(P_0)} \cdot \log |(P^n_0)'\circ \Psi_{P_0}(z)|.
\end{align*}
Note that the quantity $K \defeq \frac{d}{dt}\big|_{t=0} \delta(P_t)/\delta(P_0)$ is a real number. Therefore, there exists a constant $K \in \bbR$ such that $$\frac{d}{dt}\Big|_{t=0} \log |\lambda_t|= K \log |\lambda_0|$$ for all multipliers $\lambda_t$ of repelling periodic orbits of $P_t$. 

Reparametering $c(t)$ if necessary, the family $P_t$ has a multiplier of some repelling periodic orbit with nonzero derivative at $t=0$ since $v\not=0$. Then assuming next proposition, we derive a contradiction. 
\end{proof}

\begin{prop}\label{K=0}
Let $f$ be a hyperbolic rational map of degree at least $2$. Suppose that the Hausdorff dimension of $J(f)$ is larger than $1$. Let $\{f_t\}_{t\in I}$ be an analytic family such that $f_0=f$ and some multiplier has nonzero derivative at $t=0$. Then there is no $K\in\mathbb{R}$ such that 
$$\frac{d}{dt}\Big|_{t=0} \log |\lambda_t|= K \log |\lambda_0|$$
for all multipliers $\lambda_t$ of repelling cycles of $f_t$.
\end{prop}
\begin{proof}
Suppose there is $K\in\mathbb{R}$ such the above equation holds. We will obtain a contradiction by showing all the multipliers of repelling cycles of $f$ are real. 

Let $a_t$ be a multiplier of $f_t$ such that $a_t'|_{t=0}\not=0$.  Let $b_t$ be a repelling multiplier of $f_t$. By Corollary \ref{cor_multiplier}, for $\kappa\in(0,1)$, consider the multiplier 
$$\lambda_{t, n }=e^{i\theta_{t, n}}(a_t^{n}+a_t^{n\kappa}b_t+o(a_t^{n\kappa }b_t)).$$
Then $$\frac{d}{dt}\Big|_{t=0} \log |\lambda_{t,n}|= K \log |\lambda_{0,n}|.$$
Note that 
$$|\lambda_{t,n}|=|a_t^{n}+a_t^{n\kappa}b_t+o(a_t^{n\kappa }b_t)|=|a_t|^n\cdot|1+a_t^{n(\kappa-1)}b_t+o(a_t^{n(\kappa-1)}b_t)|.$$
To ease notation, set $\eta:=\kappa-1$. Then $\eta\in(-1,0)$. 
It follows that 
\begin{align*}
\log|\lambda_{t,n}|&=n\log|a_t|+\log |1+a_t^{n\eta}b_t+o(a_t^{n\eta}b_t)|\\
&=n\log|a_t|+\log |1+a_t^{n\eta}b_t|+\log|1+o(a_t^{n\eta}b_t)|.\\
&=n\log|a_t|+Re(a_t^{n\eta}b_t)+o(|a_t^{n\eta}b_t|).
\end{align*}
Hence 
$$\frac{d}{dt}\Big|_{t=0} \log |\lambda_{t,n}|=\frac{d}{dt}\Big|_{t=0} \log|a_t|+\frac{d}{dt}\Big|_{t=0} Re(a_t^{n\eta}b_t)+\frac{d}{dt}\Big|_{t=0}o(|a_t^{n\eta}b_t|).$$
Since $\frac{d}{dt}\Big|_{t=0} \log|a_t|=K\log|a_0|$, it follows that 
\begin{align*}
0&=\log|\lambda_{t,n}|-K\log|\lambda_0|\\
&=\frac{d}{dt}\Big|_{t=0} Re(a_t^{n\eta}b_t)+\frac{d}{dt}\Big|_{t=0}o(|a_t^{n\eta}b_t|)-K(Re(a_0^{n\eta}b_0)+o(|a_0^{n\eta}b_0|))\\
&=\frac{d}{dt}\Big|_{t=0} Re(a_t^{n\eta}b_t)-K(Re(a_0^{n\eta}b_0)+o(|a_0^{n\eta}b_0|))\\
&=Re\left(\frac{d}{dt}\Big|_{t=0}a_t^{n\eta}b_t\right)-K(Re(a_0^{n\eta}b_0)+o(|a_0^{n\eta}b_0|))\\
&=Re(n\eta a_0^{n\eta-1}a_t'|_{t=0}b_0+a_0^{n\eta}b'_t|_{t=0})-K(Re(a_0^{n\eta}b_0)+o(|a_0^{n\eta}b_0|)).
\end{align*}
Dividing by $n|a_0|^{n\eta}$ and taking limit as $n\to\infty$, we have 
$$\lim\limits_{n\to\infty}Re\left(\eta\frac{a_0^{n\eta}}{|a_0|^{n\eta}}\frac{a_t'|_{t=0}}{a_0}b_0\right)=0.$$
Since $\eta$ is real and nonzero, it follows that 
$$\lim\limits_{n\to\infty}Re\left(\frac{a_0^{n\eta}}{|a_0|^{n\eta}}\frac{a_t'|_{t=0}}{a_0}b_0\right)=0.$$

Set $u_0:=a_0/|a_0|=e^{i\theta_0}$. Then $u_0^\eta=a_0^\eta/|a_0|^\eta=e^{i\eta\theta_0}$. Choose a subsequence $n_j$ such that $u_0^{n_j\eta}\to 1$ as $j\to\infty$. Then we have 
$$\lim\limits_{j\to\infty}Re\left(u_0^{n_j\eta}\frac{a_t'|_{t=0}}{a_0}b_0\right)=Re\left(\frac{a_t'|_{t=0}}{a_0}b_0\right).$$
It follows that 
$$Re\left(\frac{a_t'|_{t=0}}{a_0}b_0\right)=0.$$

Now we claim that $a_0$ is real. For otherwise, we have $\theta_0\in(0, 2\pi)$ and $\theta_0\not=\pi$. To obtain a contradiction, we discuss in the following two cases. 

Case 1: $\theta/\pi\in(0, 2)$ is irrational. Pick $\kappa\in(0,1)$ to be rational. Then $\eta\in(-1,0)$ is rational. Choose a subsequence $n_k$ such that $u_0^{n_k\eta}\to i$ as $k\to\infty$.  Then
$$\lim\limits_{k\to\infty}Re\left(\frac{a_0^{n_k\eta}}{|a_0|^{n_k\eta}}\frac{a_t'|_{t=0}}{a_0}b_0\right)=Im\left(\frac{a_t'|_{t=0}}{a_0}b_0\right).$$
It follows that 
$$Im\left(\frac{a_t'|_{t=0}}{a_0}b_0\right)=0.$$
Thus 
$$\frac{a_t'|_{t=0}}{a_0}b_0=0.$$
Since $a_t'|_{t=0}\not=0$, we have $b_0=0$. It is impossible since $b_0$ is the multiplier of a repelling cycle of $f$.

Case 2: $\theta/\pi\in(0, 2)$ is rational and $\theta/\pi\not=1$. Pick $\kappa\in(0,1)$ to be irrational. Then $\eta\in(-1,0)$ is irrational. Write $\theta/\pi=p/q$ for two (not necessary coprime) integers $p$ and $q$ such that $u_0^q=1$. Set $n_\ell=\ell q+1$. Then $u_0^{n_\ell}=u_0$. It follows that 
\begin{align*}
0&=\lim\limits_{\ell\to\infty}Re\left(u_0^{n_\ell\eta}\frac{a_t'|_{t=0}}{a_0}b_0\right)\\
&=Re\left(u_0^\eta\frac{a_t'|_{t=0}}{a_0}b_0\right)\\
&=Re(u_0^\eta)Re\left(\frac{a_t'|_{t=0}}{a_0}b_0\right)-Im(u_0^\eta)Im\left(\frac{a_t'|_{t=0}}{a_0}b_0\right).
\end{align*}
Since $Re\left(\frac{a_t'|_{t=0}}{a_0}b_0\right)=0$, we have 
$$Im(u_0^\eta)Im\left(\frac{a_t'|_{t=0}}{a_0}b_0\right)=0.$$
Note that $Im(u_0^\eta)\not=0$. It follows that 
$$Im\left(\frac{a_t'|_{t=0}}{a_0}b_0\right)=0.$$
Then we obtain the same contradiction as in Case 1. This proves the claim.

Note that $b_t$ is an arbitrary multiplier of $f_t$ in the above argument.  If we set $b_t=a_t$, we have 
$$Re\left(\frac{a_t'|_{t=0}}{a_0}a_0\right)=0.$$
Hence $a_t'|_{t=0}$ is purely imaginary. Since $a_0$ is real, it follows that $a_t'|_{t=0}/a_0$ is purely imaginary.
 
We claim that $b_0$ is real. Indeed, it immediately follows from that $Re\left(\frac{a_t'|_{t=0}}{a_0}b_0\right)=0$ and $Re\left(\frac{a_t'|_{t=0}}{a_0}\right)=0$. 

Since $b_0$ is arbitrary, all the repelling multipliers of $f$ are real. By \cite[Theorem 1]{Eremenko11}, the Julia $J(f)$ is contained in a circle and hence the Hausdorff dimension of $J(f)$ is at most $1$. It is a contradiction. 
\end{proof}

\subsection{The component $\mathcal{H}_0$ in $\mathrm{poly}_2$}
In this subsection, we show $||\cdot||_G$ is also positive-definite on the main cardioid of the Mandelbrot set.

\begin{thm}\label{d=2}
If $\mathcal{H}_0$ is the central component in $poly_2$, then $||\cdot||_G$ is positive-definite on $\mathcal{H}_0$.
\end{thm}
\begin{proof}
    Consider a curve $P_t(z)=z^2+c(t)$ with $P_t\in\mathcal{H}_0$. We first claim that if $c(0)\not=0$, then $||v||_G \neq 0$ for any nonzero tangent vector $v \in T_{P_0}\mathcal{H}$. Indeed, if $||v||_G = 0$, then again, there exists a constant $K = \delta'(v)/\delta(0)\in\mathbb{R}$ such that $$\frac{d}{dt}\Big|_{t=0} \log |\lambda_t|= K \log |\lambda_0|$$ for all multipliers $\lambda_t$ of repelling cycles of $P_t$. Since $\delta(P_0) > 1$, Proposition \ref{K=0} gives a contradiction.

Therefore, it suffices to check that $||\cdot||_G$ is nondegenerate on the tangent space $T_{z^2}\mathcal{H}$ at $z^2$. Suppose $||v||_G = 0$ for some nonzero $v \in T_{z^2}\mathcal{H}$. Then there exists a constant $K = \delta'(v)$ such that $$\frac{d}{dt}\Big|_{t=0} \log |\lambda_t|= K \log |\lambda_0|$$ for all multipliers $\lambda_t$ of repelling cycles of $P_t$. But here $K = 0$ since $P_0(z) = z^2$ is the local minimum for the Hausdorff dimension function.  Also we note that all multipliers $\lambda_0$ of the repelling cycles of $P_0(z) = z^2$ are real numbers. Therefore, $\frac{d}{dt}\Big|_{t=0} \log |\lambda_t|=0$ implies that
$$Re\left(\frac{d}{dt}\Big|_{t=0}\lambda_t \right) = 0$$ for all multipliers $\lambda_t$ of repelling cycles of $f_t$.

The contradiction follows from direct computations. The multiplier for the repelling $1$-cycle is $1+\sqrt{1-4c(t)}$. Plugging into the above equation and using $c_0=0$, $$Re\left(\frac{d}{dt}\Big|_{t=0}\lambda_t \right) = Re\left(\frac{ - 2c'(0)}  {\sqrt{1-4c_0}(1+ \sqrt{1-4c_0})}\right) = 0$$ implies the tangent vector $v = c'(0)$ must be $\pm i$, namely the purely imaginary direction. On the other hand, there are two $3$-cycles and their multipliers are \\$-4\left(-c(t)-2 \pm c(t)\sqrt{-4c(t)-7}\right)$, respectively.
Therefore, $$Re\left(\frac{d}{dt}\Big|_{t=0}\lambda_t \right) = Re\left( \frac{-c'(0) + c'(0)\sqrt{-4c_0-7} - \frac{2c_0c'(0)}{\sqrt{-4c_0-7}}}{-c_0-2 + c_0\sqrt{-4c_0-7}} \right).$$
But $c_0 =0$ and $v = \pm i$ do not give $Re\left(\frac{d}{dt}\big|_{t=0}\lambda_t \right) = 0$, which is a contradiction.

Hence, $||\cdot||_G$ is positive-definite on $\mathcal{H}_0$.
\end{proof}

If $\mathcal{H}_0$ is the central component in $\mathrm{poly}_d$ for $d \ge 3$, then by the same argument as in Theorem \ref{d=2}, the form $||\cdot||_G$ is positive-definite on the tangent space $T_{[P]}\mathcal{H}_0$ if $[P]\not=[z^d]$. Therefore the positive-definiteness of $||\cdot||_G$ on $\mathcal{H}_0$ is reduced to the positive-definiteness of $||\cdot||_G$ on the tangent space $T_{[z^d]}\mathcal{H}_0$. However, the proof of Theorem \ref{d=2} is much difficult to reproduce for $T_{[z^d]}\mathcal{H}_0$ when $d \ge 3$. In fact, the positive-definiteness of $||\cdot||_G$ on $T_{[z^d]}\mathcal{H}_0$ is equivalent to a negative answer of the following question.
\begin{question}
For $d\ge 3$, let $\{P_t\}_{t\in I}$ be an analytic family with $P_0(z)=z^d$. Are $\lambda_t'|_{t=0}$ purely imaginary for all repelling multipliers $\lambda_t$ of $P_t$?
\end{question}

\section{Applications to the Hausdorff dimension of Julia sets} \label{sec_nolocmax}
We give two applications of our metric to the Hausdorff dimension of Julia sets. First, we show that the Hausdorff dimension function $\delta :\mathcal{H} \to \bbR$ has no local maximum on any hyperbolic component $\mathcal{H}$ in $\mathrm{poly}_d\setminus\mathcal{S}_d$. This result relates to a theorem due to Ransford \cite{Ransford93} where he proved the result for an analytic family of degree $d \ge 2$ rational maps parametrized by a simply connected domain in $\bbC$. Second, we give a sufficient condition for a point not being a critical point of $\delta$.

We begin with the first application.
\begin{thm}\label{thm:no-max}
	Let $\mathcal{H}$ be a hyperbolic component in $\mathrm{poly}_d\setminus\mathcal{S}_d$. The Hausdorff dimension function $\delta: \mathcal{H} \to [1,2)$ has no local maximum on $\mathcal{H}$.
\end{thm}

\begin{proof}
Note if $\mathcal{H}=\mathcal{H}_0$ is the central component, then $[z^d]\in\mathcal{H}$ is a local minimum of $\delta$. Now for a hyperbolic component $\mathcal{H}\subset\mathrm{poly}_d\setminus\mathcal{S}_d$, let $[P_0]\in\mathcal{H}$ with $P_0\in\mathrm{Poly}_d$  and $P_0\not=z^d$. We can further assume that $P_0$ is monic and centered. Set $\nu$ the equilibrium state of $-\delta(P_0)\log|P_0'(z)|$. We first show that $[P_0]$ is a critical point of $\delta$ if and only if $P_0$ is a critical point of $M_{\nu}$.
	Suppose $[P_0]\in\mathcal{H}$ is a critical point of $\delta$. Then $\delta'(v) =0$ for all $v \in T_{[P_0]}\mathcal{H}$. 
	Now denote by $\mathrm{Poly}_d^\ast$ the space of the monic and centered polynomials of degree $d$ and consider a lift $\widetilde{\mathcal{H}}\subset\mathrm{Poly}_d^\ast$ of $\mathcal{H}$ such that $P_0\in\widetilde{\mathcal{H}}$. Then $\delta'(\tilde{v}) =0$ for all $\tilde{v} \in T_{P_0}\widetilde{\mathcal{H}}$ Moreover, $(\delta M_{\nu})'(\tilde{v}) = 0$ for all $\tilde{v} \in T_{P_0}\widetilde{\mathcal{H}}$  since $P_0$ is a local minimum for $G_{P_0} = \delta M_{\nu}$ by Proposition \ref{prop_metric}. 
	Then by chain rule and using $\delta'(\tilde{v}) = 0$, we have
	$$0 = (\delta M_{\nu})'(\tilde{v}) = \delta'(\tilde{v}) M_{\nu}(P_0) + \delta(P_0) M_{\nu}'(\tilde{v}) = \delta(P_0) M_{\nu}'(\tilde{v}).$$
	Since $\delta(P_0) \neq 0$, we have $M_{\nu}'(\tilde{v}) =0$ for all $\tilde{v} \in T_{P_0}\widetilde{\mathcal{H}}$, i.e. $M_{\nu}$ has a critical point at $P_0$. The converse direction works the same way.
	
	Recall the complex multiplier function $$\widetilde{M}_{\nu}(P) = \int_{J(P_0)} \log P'(z) d\nu.$$
	Then $M_{\nu}(P) = Re(\widetilde{M}_{\nu}(P))$ and
	$$Re(\widetilde{M}_{\nu}'(\tilde{v})) = M_{\nu}'(\tilde{v}) =0 \text{ for all } \tilde{v} \in T_{P_0}\widetilde{\mathcal{H}}.$$
	Let $\tilde{w} = J \cdot \tilde{v} \in T_{P_0}\widetilde{\mathcal{H}}$ where $J$ is the complex structure on the tangent space $T_{P_0}\widetilde{\mathcal{H}}$. Then
	$$0 = Re(\widetilde{M}_{\nu}'(\tilde{w})) = Re(i \cdot \widetilde{M}_{\nu}'(\tilde{v})) = -Im(\widetilde{M}_{\nu}'(\tilde{v})).$$
	Hence, $M_{\nu}'(\tilde{v}) = 0$ for all $\tilde{v} \in T_{P_0}\widetilde{\mathcal{H}}$ and therefore $\widetilde{M}_{\nu}''$ defines a {\it complex} bilinear $2$-form on $T_{P_0}\widetilde{\mathcal{H}}$.
	
	Since $M_{\nu}''$ is the real part of the form $\widetilde{M}_{\nu}''$, then $M_{\nu}''$ is a real $2$-form. Denote $V_0,V_+$ and $V_-$ the subspace of $T_{P_0}\widetilde{\mathcal{H}}$ on which $M_{\nu}''$ is zero, positive-definite and negative-definite respectively. Then $T_{P_0}\mathcal{M} = V_0 \oplus V_+ \oplus V_-$. Since the complex structure $J$ is an isomorphism between $V_+$ and $V_-$, we have that $dim_\mathbb{R}(V_+) = dim_\mathbb{R}(V_-)$. Hence $dim_\mathbb{R}(V_-) + dim_\mathbb{R}(V_0) \ge d-1$.
	
	Note that $$(\delta M_{\nu})'' = \delta'' M_{\nu} + 2\delta' M_{\nu}' + \delta M_{\nu}'' = \delta'' M_{\nu} + \delta M_{\nu}''.$$ From the previous section and the fact that  $\widetilde{\mathcal{H}}$ is a branched finite cover of $\mathcal{H}$ with the only branched point at $z^d$, the $2$-form $(\delta M_{\nu})''$ is positive-definite on $T_{P_0}\widetilde{\mathcal{H}}$. Therefore, the Hessian $\delta''$ of Hausdorff dimension must be positive-definite on $V_0 \oplus V_-$, which implies that $\delta''$ has positive-definite $\mathbb{R}$-dimension at least $d-1$. Then $P_0$ cannot be a local maximum of $\delta$ on $\widetilde{\mathcal{H}}$, for otherwise $\delta''$ should have positive-definite $\mathbb{R}$-dimension $0$. It follows that $[P_0]$ is not a local maximum of $\delta$ on $\mathcal{H}$.
	\end{proof}

More generally, let $\mathcal{H} \subset\mathrm{rat}_d$ be a hyperbolic component. We explain how our methods developed in the previous sections can be used to give a sufficient condition for a point in $\mathcal{H}$ not being a critical point of $\delta : \mathcal{H} \to (0,2)$.

\begin{thm}\label{thm_suffcond}
	Let $\mathcal{H} \subset\mathrm{rat}_d$ be a hyperbolic component. Then $[f_0] \in\mathcal{H}$ is not a critical point of $\delta$ if 
	\begin{align}
	\inf_{\{\hat{x}_n\}_{n \ge 1}} \displaystyle \liminf_{n \to \infty} \frac{1}{n} \left|\frac{D\lambda_{\hat{x}_n}([f_0])}{\lambda_{\hat{x}_n}([f_0])}\right| \neq 0
	\end{align}
	where $\hat{x}_n$ is an $n$-cycle of $f_0$ and $\lambda_{\hat{x}_n}: \mathcal{H} \to \bbC$ is the function sending $[f]$ to the multiplier of $\phi_f(\hat{x}_n)$ where $f_0$ and $f$ are in the same hyperbolic component in $Rat_d$. The map $D\lambda_{\hat{x}_n}([f_0]): T_{[f_0]}\mathcal{H} \to \bbC$ is the differential of $\lambda_{\hat{x}_n}$. The infimum is taken over all the sequences of cycles $\hat{x}_n$.
\end{thm}

Before proving the theorem, we state two elementary lemmas.
\begin{lem}\label{lem:not-crit}
The following are equivalent.
\begin{enumerate}
\item The point $[f_0] \in\mathcal{H}$ is not a critical point of $\delta$.
\item Let $\widetilde{\mathcal{H}} \subset\mathrm{Rat}_d$ be a lift of $\mathcal{H}$ containing $f_0$. There exists $\tilde{v}^* \in T_{f_0}\widetilde{\mathcal{H}}$ such that $M_{\nu}'(\tilde{v}^*) \neq 0$. Here $\nu$ is the equilibrium state for $-\delta(f_0)\log|f_0'|$.
\end{enumerate}
\end{lem}
\begin{proof}
The proof follows from the same argument as in the first paragraph of the proof of Theorem \ref{thm:no-max}.
\end{proof}

Furthermore, since primitive orbit measures are dense (Proposition \ref{density}), for $f_0\in \widetilde{\mathcal{H}}$, there exists a sequence $\mu_k$ of primitive orbit measures such that $\mu_k$ converges to the the equilibrium state $\nu$ of $-\delta(f_0)\log|f_0'(z)|$. Denote by $\hat{x}_{n_k}$ the periodic cycle which is the support of $\mu_k$.
Let $U(f_0)$ be as in Section \ref{intro}. For $f \in U(f_0)$, define
$$L_k(f):= \int_{J(f)} \log |f'| ~d(\phi_f)_*\mu_k= \frac{1}{n_k} \log |\lambda(\phi_f(\hat{x}_{n_k}))|,$$
where $\lambda(\phi_f(\hat{x}_{n_k}))$ is the multiplier of the cycle $\phi_f(\hat{x}_{n_k})$ under $f$.

\begin{lem} \label{lem_gnM}
	For any $\tilde{v} \in T_{f_0}\widetilde{\mathcal{H}}$, we have $\displaystyle \lim_{k \to \infty} L_k'(\tilde{v}) = M_{\nu}'(\tilde{v})$.
\end{lem}
\begin{proof}
	By definition of convergence in $\Omega(J(f_0))$, we have that $L_k(f)$ converges pointwise to $M_{\nu}(f)$. Now we show that the sequence $L_k(f)$ converges uniformly on compact subsets of $U(f_0)$ to $M_{\nu}(f)$. We only need to show that $L_k(f)$ is uniformly bounded on compact subsets of $U(f_0)$. The uniform boundedness follows from	
	$$L_k(f) = \frac{1}{n_k} \log |(f^{n_k})'(x)| \le\max\{\log|f'(z)|: z\in J(f)\}$$
	and compactness of $J(f)$. It follows that $ L_k$ converges locally uniformly to $M_{\nu}$. Hence the conclusion follows.
\end{proof}

\begin{proof}[Proof of Theorem \ref{thm_suffcond}]
	Let $\nu$ and $\mu_k$ be as above. By Lemmas \ref{lem:not-crit} and \ref{lem_gnM}, it suffices to show there exists $\tilde{v}^\ast\in T_{f_0}\widetilde{\mathcal{H}}$ such that $\displaystyle \lim_{k \to \infty} L_k'(\tilde{v}^\ast)\not=0$. Note that 
	$$L_k'(\tilde{v}^\ast)=\frac{1}{n_k} \frac{|\lambda_{\hat{x}_{n_k}}([f_0])|'(\tilde{v}^\ast)}{|\lambda_{\hat{x}_{n_k}}([f_0])|}.$$
	If the condition $(7.1)$ holds, we have that there is $\tilde{v}_{n_k}\in T_{P_0}\widetilde{\mathcal{H}}$ such that
	$$\liminf\limits_{k\to\infty}\frac{1}{n_k} \frac{|\lambda'_{\hat{x}_{n_k}}([f_0])(\tilde{v}_{n_k})|}{|\lambda_{\hat{x}_{n_k}}([f_0])|}\not=0.$$
	Since $T_{P_0}\mathcal{H}$ is finite dimensional, there exists $\tilde{v}^\ast\in T_{f_0}\widetilde{\mathcal{H}}$ such that 
	$$\liminf\limits_{k\to\infty}\frac{1}{n_k} \frac{|\lambda'_{\hat{x}_{n_k}}([f_0])(\tilde{v}^\ast)|}{|\lambda_{\hat{x}_{n_k}}([f_0])|}\not=0.$$
	It follows that 
	$$\liminf\limits_{k\to\infty}\frac{1}{n_k} \frac{|\lambda_{\hat{x}_{n_k}}([f_0])|'(\tilde{v}^\ast)}{|\lambda_{\hat{x}_{n_k}}([f_0])|}\not=0.$$
	Passing to a subsequence if necessary, we have that 
	$\lim\limits_{k \to \infty} L_k'(\tilde{v}^\ast)\not=0$. This completes the proof.
	
\end{proof}

\bibliographystyle{siam}
\bibliography{references}

\begin{thebibliography}{10}

\bibitem{Bodart96}
{\sc O.~Bodart and M.~Zinsmeister}, {\em Quelques r\'{e}sultats sur la
  dimension de {H}ausdorff des ensembles de {J}ulia des polyn\^{o}mes
  quadratiques}, Fund. Math., 151 (1996), pp.~121--137.

\bibitem{Bowen79}
{\sc R.~Bowen}, {\em Hausdorff dimension of quasicircles}, Inst. Hautes
  \'{E}tudes Sci. Publ. Math.,  (1979), pp.~11--25.

\bibitem{Bridgeman10}
{\sc M.~Bridgeman}, {\em Hausdorff dimension and the {W}eil-{P}etersson
  extension to quasifuchsian space}, Geom. Topol., 14 (2010), pp.~799--831.

\bibitem{Bridgeman08}
{\sc M.~J. Bridgeman and E.~C. Taylor}, {\em An extension of the
  {W}eil-{P}etersson metric to quasi-{F}uchsian space}, Math. Ann., 341 (2008),
  pp.~927--943.

\bibitem{Douady97}
{\sc A.~Douady, P.~Sentenac, and M.~Zinsmeister}, {\em Implosion parabolique et
  dimension de {H}ausdorff}, C. R. Acad. Sci. Paris S\'{e}r. I Math., 325
  (1997), pp.~765--772.

\bibitem{Eremenko11}
{\sc A.~Eremenko and S.~van Strien}, {\em Rational maps with real multipliers},
  Trans. Amer. Math. Soc., 363 (2011), pp.~6453--6463.

\bibitem{Havard00}
{\sc G.~Havard and M.~Zinsmeister}, {\em Thermodynamic formalism and variations
  of the {H}ausdorff dimension of quadratic {J}ulia sets}, Comm. Math. Phys.,
  210 (2000), pp.~225--247.

\bibitem{He18}
{\sc Y.~M. He}, {\em Basmajian-type identities and {H}ausdorff dimension of
  limit sets}, Ergodic Theory Dynam. Systems, 38 (2018), pp.~2224--2244.

\bibitem{Jaksztas11}
{\sc L.~Jaksztas}, {\em On the derivative of the {H}ausdorff dimension of the
  quadratic {J}ulia sets}, Trans. Amer. Math. Soc., 363 (2011), pp.~5251--5291.

\bibitem{Oliver02}
{\sc O.~Jenkinson and M.~Pollicott}, {\em Calculating {H}ausdorff dimensions of
  {J}ulia sets and {K}leinian limit sets}, Amer. J. Math., 124 (2002),
  pp.~495--545.

\bibitem{Kiwi15}
{\sc J.~Kiwi}, {\em Rescaling limits of complex rational maps}, Duke Math. J.,
  164 (2015), pp.~1437--1470.

\bibitem{Makarov85}
{\sc N.~G. Makarov}, {\em On the distortion of boundary sets under conformal
  mappings}, Proc. London Math. Soc. (3), 51 (1985), pp.~369--384.

\bibitem{McMullen94}
{\sc C.~T. McMullen}, {\em Complex dynamics and renormalization}, vol.~135 of
  Annals of Mathematics Studies, Princeton University Press, Princeton, NJ,
  1994.

\bibitem{McMullen98}
\leavevmode\vrule height 2pt depth -1.6pt width 23pt, {\em Hausdorff dimension
  and conformal dynamics. {III}. {C}omputation of dimension}, Amer. J. Math.,
  120 (1998), pp.~691--721.

\bibitem{McMullen00}
\leavevmode\vrule height 2pt depth -1.6pt width 23pt, {\em Hausdorff dimension
  and conformal dynamics. {II}. {G}eometrically finite rational maps}, Comment.
  Math. Helv., 75 (2000), pp.~535--593.

\bibitem{McMullen08}
\leavevmode\vrule height 2pt depth -1.6pt width 23pt, {\em Thermodynamics,
  dimension and the {W}eil-{P}etersson metric}, Invent. Math., 173 (2008),
  pp.~365--425.

\bibitem{Oh17}
{\sc H.~Oh and D.~Winter}, {\em Prime number theorems and holonomies for
  hyperbolic rational maps}, Invent. Math., 208 (2017), pp.~401--440.

\bibitem{Parry90}
{\sc W.~Parry and M.~Pollicott}, {\em Zeta functions and the periodic orbit
  structure of hyperbolic dynamics}, Ast\'{e}risque,  (1990), p.~268.

\bibitem{Przytycki06}
{\sc F.~Przytycki}, {\em On the hyperbolic {H}ausdorff dimension of the
  boundary of a basin of attraction for a holomorphic map and of
  quasirepellers}, Bull. Pol. Acad. Sci. Math., 54 (2006), pp.~41--52.

\bibitem{Ransford93}
{\sc T.~J. Ransford}, {\em Variation of {H}ausdorff dimension of {J}ulia sets},
  Ergodic Theory Dynam. Systems, 13 (1993), pp.~167--174.

\bibitem{Ruelle82}
{\sc D.~Ruelle}, {\em Repellers for real analytic maps}, Ergodic Theory
  Dynamical Systems, 2 (1982), pp.~99--107.

\bibitem{Ruelle89}
\leavevmode\vrule height 2pt depth -1.6pt width 23pt, {\em The thermodynamic
  formalism for expanding maps}, Comm. Math. Phys., 125 (1989), pp.~239--262.

\bibitem{Ruelle04}
\leavevmode\vrule height 2pt depth -1.6pt width 23pt, {\em Thermodynamic
  formalism}, Cambridge Mathematical Library, Cambridge University Press,
  Cambridge, second~ed., 2004.
\newblock The mathematical structures of equilibrium statistical mechanics.

\bibitem{Sigmund74}
{\sc K.~Sigmund}, {\em On dynamical systems with the specification property},
  Trans. Amer. Math. Soc., 190 (1974), pp.~285--299.

\bibitem{Zdunik90}
{\sc A.~Zdunik}, {\em Parabolic orbifolds and the dimension of the maximal
  measure for rational maps}, Invent. Math., 99 (1990), pp.~627--649.

\bibitem{Zinsmeister00}
{\sc M.~Zinsmeister}, {\em Thermodynamic formalism and holomorphic dynamical
  systems}, vol.~2 of SMF/AMS Texts and Monographs, American Mathematical
  Society, Providence, RI; Soci\'{e}t\'{e} Math\'{e}matique de France, Paris,
  2000.
\newblock Translated from the 1996 French original by C. Greg Anderson.

\end{thebibliography}
\end{document}